\documentclass[preprint,12pt,3p]{elsarticle}

\usepackage{amssymb}
\usepackage{amsthm}

\usepackage{amsmath}
\usepackage{amsfonts}
\usepackage{subfigure}
\usepackage{graphicx}
\usepackage{color}
\usepackage{enumerate}
\usepackage[colorlinks=true]{hyperref}
\newtheorem{theorem}{Theorem}[section]
\newtheorem{proposition}[theorem]{Proposition}
\newtheorem{lemma}[theorem]{Lemma}

\newtheorem{corollary}[theorem]{Corollary}

\theoremstyle{definition}

\newtheorem{algorithm}[theorem]{Algorithm}
\newtheorem{construction}[theorem]{Construction}
\theoremstyle{remark}
\newtheorem*{remark}{Remark}
\newtheorem*{assertion}{Assertion}
\newtheorem*{claim}{Claim}
\numberwithin{equation}{section}
\newenvironment{decisionproblem}[1]%
	{\vskip\topsep\noindent\textsc{#1.\/}}%
	{\par\vskip\topsep}%
\def\xp#1{\par\global\leftskip=#1\parindent
       \noindent\hskip-\parindent\ignorespaces}% exdented (hanging)
\def\xpend{\par\global\leftskip=0pt\noindent\ignorespaces}
\def\ga{N}
\def\gb{P}
\def\gc{Q}
\def\gd{R}
\def\gs{A}
\def\gt{B}
\def\va(#1#2){a_{#1#2}}
\def\vb(#1#2){b_{#1#2}}
\def\vc(#1#2){c_{#1#2}}
\def\vd(#1#2){d_{#1#2}}
\def\vpz{p_0}
\def\vp{p}
\def\vxa{x_1}
\def\vya{y_1}
\def\vxb{x_2}
\def\vyb{y_2}
\def\vxc{x_3}
\def\vyc{y_3}
\def\vxd{x_4}
\def\vyd{y_4}
\def\vxe{x_5}
\def\vye{y_5}
\def\vq{q}
\def\vqz{q_0}
\def\ww#1#2#3{W_{#1}^{#2#3}}
\def\wb#1#2#3{X_{#1}^{#2#3}}
\def\ws#1{Y_{#1}}
\def\wt#1#2#3{Z_{#1}^{#2#3}}
\def\wto#1{Z_{#1}}
\def\rev(#1){{#1}^{-1}}
\def\revb(#1){({#1})^{-1}} % with brackets
\def\db(#1#2){\delta_{#1}({#2})}
\def\ein(#1#2){\sigma^-_{#1}({#2})}
\def\eout(#1#2){\sigma^+_{#1}({#2})}
\def\autb(#1#2){\pi_{#1#2}}
\def\auts{\alpha}
\def\siff{\Leftrightarrow}
\def\simplies{\Rightarrow}

\let\isom\cong % isomorphic
\def\cc{\mkern0.2\thinmuskip{\cdot}\mkern0.2\thinmuskip}%
\def\eth{{E_3}}
\def\ethg{{E_3^G}}
\def\ethgc{{E_3^{\gc}}}
\def\Rot{{\fam0 Rot}}
\def\cwn{{\fam0 cwn}}
\def\cp{{\fam0 cp}}
\def\srs{{\fam0 srs}}
\def\ees{D} % set of edge-ends

\def\realline{\hbox to \hsize}% plain TeX \line
\journal{Theoretical Computer Science}
\makeatletter
\def\ps@pprintTitleNoSubTo{%
     \let\@oddhead\@empty
     \let\@evenhead\@empty
     \def\@oddfoot{\footnotesize\itshape\hfill\today}%
     \let\@evenfoot\@oddfoot}
\let\ps@pprintTitle\ps@pprintTitleNoSubTo% uncomment for no "Sub to ..."
\makeatother

\begin{document}

\begin{frontmatter}

\title{Edge-outer graph embedding and the complexity of the DNA reporter strand problem} %\texttt{elsarticle} class\tnoteref{label0}}
%\tnotetext[label0]{This is only an example}

\author[label1, label3]{M. N. Ellingham\corref{cor1}}

\address[label1]{Department of Mathematics, 1326 Stevenson Center,
Vanderbilt University,
Nashville, Tennessee 37240}
\fntext[label3]{Supported by Simons Foundation award 429625}

\ead{mark.ellingham@vanderbilt.edu} \ead[url]{https://math.vanderbilt.edu/ellingmn/}

\author[label2,label4]{Joanna A. Ellis-Monaghan}

\address[label2]{Department of Mathematics, Saint Michael's College, One Winooski Park, Colchester, Vermont 05458}
\ead{jellis-monaghan@smcvt.edu}
\ead[url]{http://www.smcvt.edu/pages/get-to-know-us/faculty/ellis-monaghan-jo.aspx}
\fntext[label4]{Supported by NSF grant DMS-1332411}

\begin{abstract}
In 2009, Jonoska, Seeman and Wu showed that every graph admits a route for a DNA reporter strand, that is, a closed walk covering every edge either once or twice, in opposite directions if twice, and passing through each vertex in a particular way.  This corresponds to showing that every graph has an \emph{edge-outer embedding}, that is, an orientable embedding with some face that is incident with every edge. In the motivating application, the objective is such a closed walk of minimum length.  Here we give a short algorithmic proof of the original existence result, and also prove that finding a shortest length solution is NP-hard, even for $3$-connected cubic ($3$-regular) planar graphs. Independent of the motivating application, this problem opens a new direction in the study of graph embeddings, and we suggest several problems emerging from it.

\end{abstract}

\begin{keyword} %% keywords here, in the form: keyword \sep keyword example
DNA origami \sep reporter strand \sep orientable graph embedding
\sep one-face embedding \sep edge-outer embedding %% MSC codes here, in the form: \MSC code \sep code
%% or \MSC[2008] code \sep code (2000 is the default)
\MSC 05C10 \sep 05C90 \sep 68Q17 \sep 92D10

\end{keyword}

\end{frontmatter}

%%
%% Start line numbering here if you want
%%
% \linenumbers

\section{Introduction}
\label{sec:intro}

DNA self-assembly, and self-assembly in general, is a rapidly advancing field, with \cite{P07, S15} providing good overviews. 	In 2006, Rothemund introduced `DNA origami', a new self-assembly method that increased the scale of DNA constructs and is one of the major developments in DNA nanotechnology this century~\cite{rothemund-folding-2006}.  DNA origami originally involved combining an M13 single-stranded cyclic viral molecule, called the \emph{scaffolding strand}, with 200-250 short staple strands to produce a 90 $\times$ 90 nm tile (in 2D), but now these strands can also produce 3D constructs with the structure of graphs or graph fragments \cite{D+09}.  At its most basic level, the design objective for DNA origami assembly of a graph-like structure is a strategy with the scaffolding strand following a single walk that traverses every edge at least once, with any edges that are traversed more than once visited exactly twice, in opposite directions (because DNA strands in a double helix are oppositely directed), and without separating or crossing through at a vertex.  See \cite{B+15, EMMP14, EPSetal17} for further work on routing scaffolding strands.

The problem of finding a similarly prescribed walk also arises in the context of determining an efficient route for a \emph{reporter strand}, that is, a strand that is recovered and read at the end of an experiment to report on one of the products of the experiment.  In designing the DNA self-assembly of a molecule with the structure of a graph $G$,  the boundary components of a `thickened' version of $G$ identify the circular DNA strands that assemble (hybridize) into the graph $G$.
 For details, see \cite{JSW09}, where the objective was to show that every graph has an associated orientable thickened graph, with a boundary component visiting every edge at least once, thus corresponding to the desired route for the reporter strand.

While motivated by a particular application, this problem of finding suitable walks is of independent intrinsic interest in topological graph theory.
Thickened graphs are also known as \emph{ribbon graphs}, and are equivalent to embeddings of graphs in compact surfaces.  All embeddings in this paper will be \emph{cellular}, in which every face is homeomorphic to an open disk.
Each face of an embedding corresponds to a boundary component of a thickened graph, which corresponds to a circular strand of DNA.
Thus, showing the existence of a suitable walk for a scaffolding  or reporter strand is equivalent to proving that every graph admits an orientable embedding where every edge lies on a single face.  The facial walk of this face gives the corresponding desired route for the DNA strand.  Prompted by this application, we define a \emph{reporter strand walk} in a graph $G$ to be a walk that uses every edge of $G$ at least once and occurs as a facial boundary walk in some orientable embedding of $G$.
See Figure \ref{fig:rsfw} for two examples of embeddings of $K_4$ on the torus with facial walks that are reporter strand walks.  Notice that the walk shown on the right is shorter (has fewer edges) than the one on the left.

Having a face that includes every edge is intermediate between two well-known properties of graph embeddings.  A \emph{one-face embedding} is an embedding in which there is only one face, so every edge occurs exactly twice on the boundary of this single face; a one-face embedding is necessarily a maximum genus embedding.  An \emph{outer embedding} in a given surface is an embedding in which all vertices appear on a single face (outerplanar graphs are particularly well-studied).  It therefore seems appropriate to call an embedding in which there is a face (the `outer' face) that includes every edge an \emph{edge-outer embedding}.  For connected graphs, a one-face embedding is edge-outer, and an edge-outer embedding is outer.

\begin{figure}[t]
\realline{\hfill%
        \includegraphics[scale=.3]{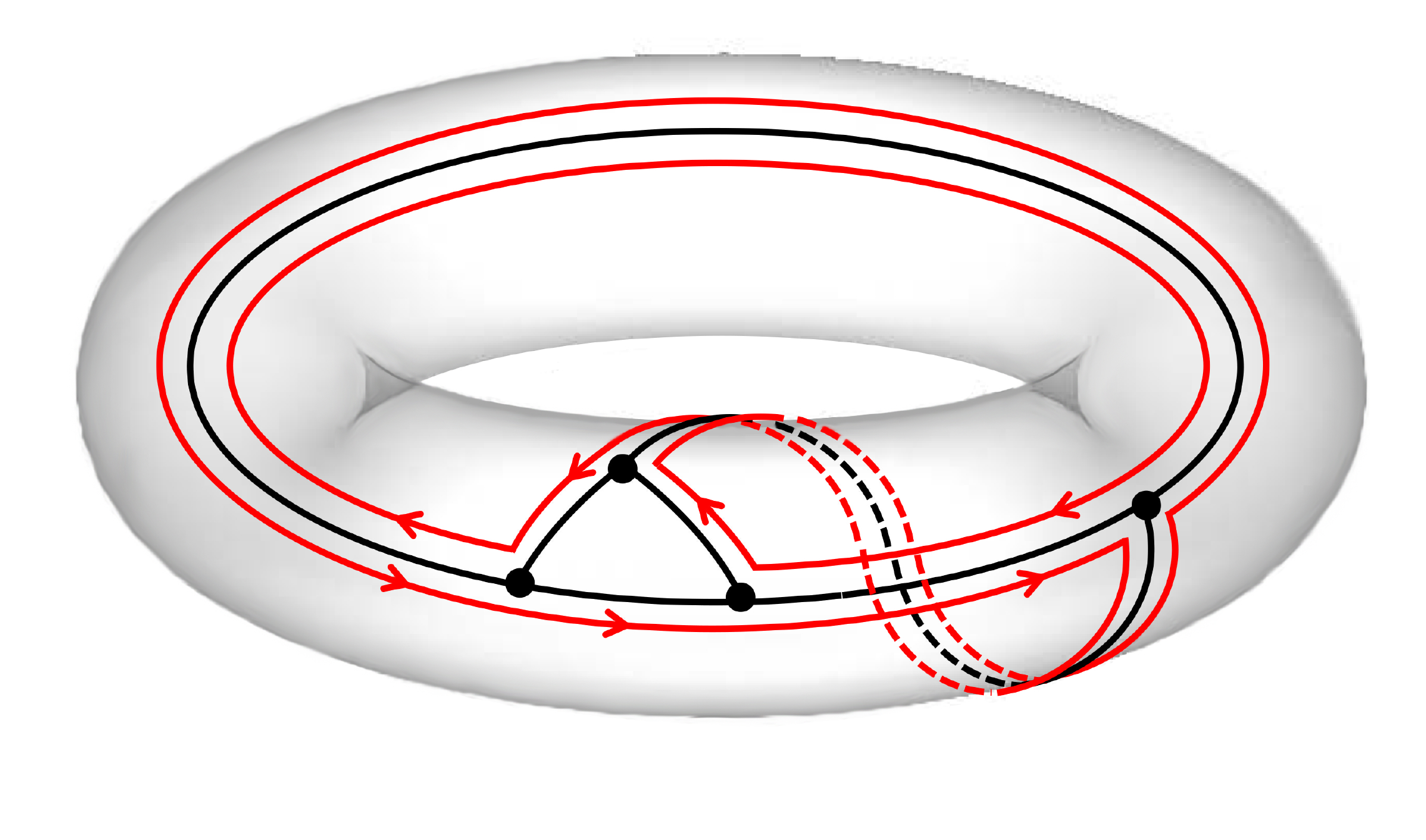}%
        \hfill
        \includegraphics[scale=.3]{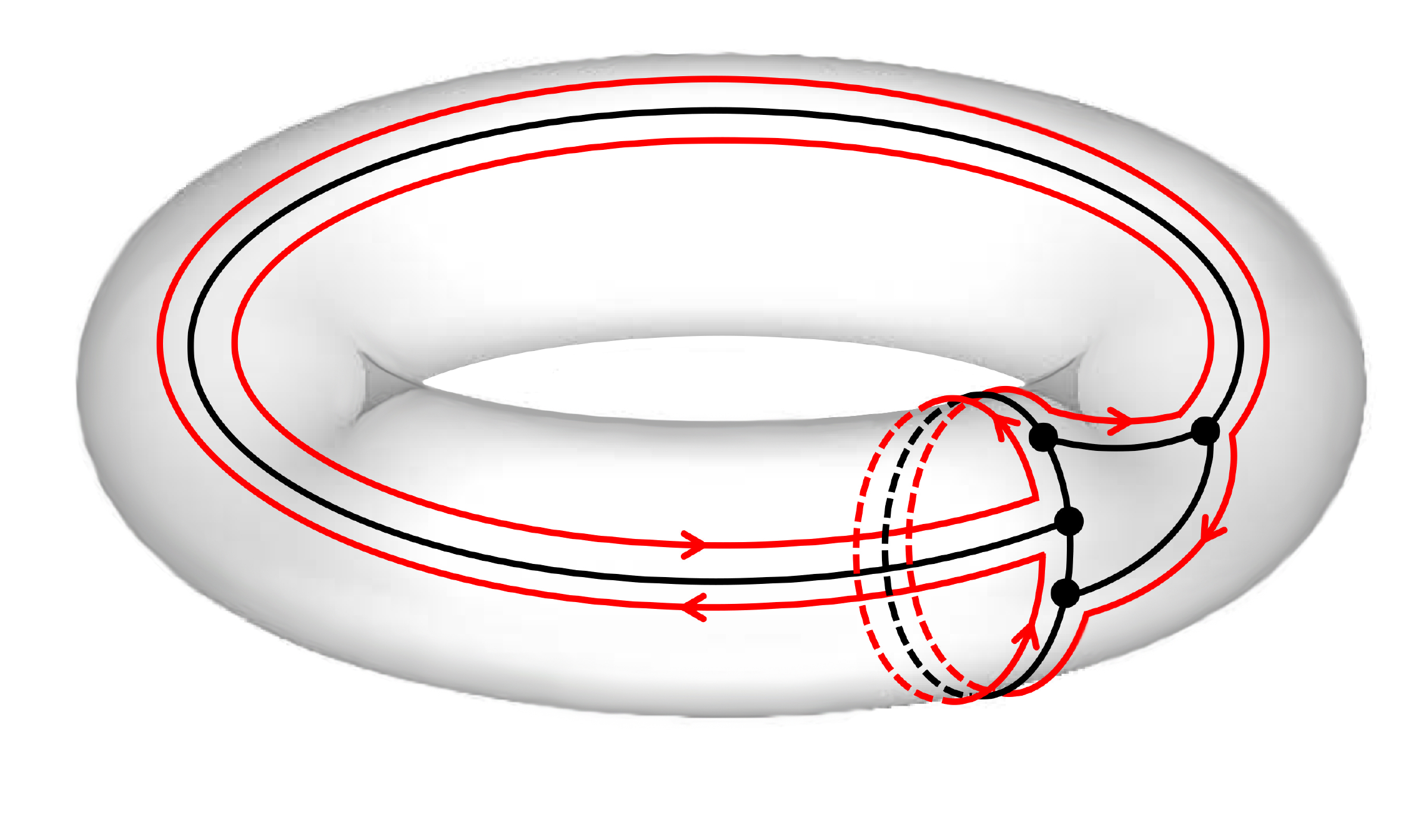}%
        \hfill}
\caption{Facial walks corresponding to reporter strands.}
\label{fig:rsfw}
\end{figure}

Here we give a short proof of the result from \cite{JSW09} that reporter strand walks always exist, and hence every graph has an edge-outer embedding.  Our result provides a polynomial-time algorithm that finds both the embedding and the reporter strand walk.
Furthermore, we show that the problem of finding a shortest length reporter strand walk, or equivalently, an embedding with smallest degree outer face, is NP-hard, even for $3$-connected cubic planar graphs.  We begin with a weaker NP-hardness result in Section \ref{sec:2c} and strengthen the result in Section \ref{sec:3c}.

\section{A short proof that reporter strand walks exist}
\label{sec:shortproof}

In this section we provide a short proof of the existence of reporter strand walks in all graphs.
In this paper, graphs may have loops and multiple edges.  A graph with neither loops nor multiple edges is \emph{simple}.  Sometimes we think of an edge as consisting of two distinct \emph{edge-ends} or just \emph{ends};
$\ees_G(v)$ denotes the set of ends incident with $v$ in $G$.  We refer the reader to \cite{West} for graph theory terms not defined in this paper.  Most terms defined explicitly in Sections \ref{sec:shortproof} to \ref{sec:3c} are terms we introduce for our specific needs in this paper.

We assume the reader is familiar with combinatorial descriptions of orientable cellular embeddings of graphs, using \emph{rotation schemes} or \emph{rotation systems}, as described in \cite{E-MM,GT,MT}.
A rotation scheme assigns to each vertex a \emph{rotation}, which is a cyclic ordering of the edge-ends incident with that vertex, corresponding to their order in the globally consistent clockwise direction in the surface.
Every orientable embedding is determined up to homeomorphism by its rotation scheme.

Suppose we have an embedding of a graph $G$ described by a rotation scheme.  Suppose also that at some vertex $v$ there is an incident face $f$ and an incident edge-end $d$, so that $d$ is not incident with $f$.  Assume that $f$ appears between consecutive edge-ends $d', d''$ in the rotation at $v$ ($f$ may also appear in other places around $v$).  Let $d^-$, $d^+$ be the edge-ends immediately before and after $d$ in the rotation at $v$, respectively.  Then the operation of \emph{flipping $d$ into $f$ (between $d'$ and $d''$)} modifies the rotation at $v$ from $(d^-, d, d^+, \dots, d', d'', \dots)$ to $(d^-, d^+, \dots, d', d, d'', \dots)$, moving $d$ into a position between $d'$ and $d''$.  See Figure~\ref{fig:rsalg} (which is explained in more detail below).

The operation of moving one edge-end in a vertex rotation was used, for example, by Duke \cite[Theorem 3.2]{Duke66} to show that the orientable genus range for a graph forms an interval.  Jonoska, Seeman and Wu \cite[Figure 5]{JSW09} used the special case of this operation for cubic graphs.

In our terminology, the main result of \cite{JSW09} is that every connected graph has a reporter strand walk.  The following algorithm provides a short proof.

\begin{algorithm}
\label{thm:rsalg}
Given a connected graph $G$ (with loops and multiple edges allowed):
\smallskip
\xp2 Take an arbitrary orientable embedding of $G$.
\xp2 Choose an arbitrary face $f$.
\xp2 While some edge is not in $f$ $\{$
   \xp3 Choose an edge $e$ not in $f$, but incident with
       a vertex $v$ of $f$.
   \xp3 Modify the embedding by flipping an end of $e$ incident with $v$ into $f$.
   \xp3 There is a new face using $e$ twice; let $f$ be this face.
\xp2 $\}$
\smallskip
\xpend
 \begin{assertion}
This algorithm runs in polynomial time.  It terminates with an orientable embedding of $G$ in which $f$ is a face using every edge of $G$.  Thus, the boundary walk of $f$ is a reporter strand walk in $G$.
 \end{assertion}
\end{algorithm}

\begin{proof}[Proof that the algorithm works]
The initial embedding exists because we may just give each vertex an arbitrary rotation.
The edge $e$ always exists because $G$ is connected.  When we flip the end of $e$ into $f$, we create a new face $f'$ that includes all edges of the old face $f$, and uses $e$ twice.  (If $e$ belonged to two distinct old faces then $f'$ also uses all other edges from those faces.
Otherwise, $e$ belonged to a single old face $g$, and the two occurrences of $e$ split the boundary of $g$ into two pieces; $f'$ includes the edges of one of those pieces, and the other piece becomes the boundary of a second new face.  Also, the length of the face $f$ increases both when $e$ has only one end incident with a vertex of $f$ and when it has two, as the example below illustrates.)
Since $f'$ becomes the new $f$, the edge set of $f$ strictly increases at each iteration, until it contains all edges.

Since each edge is flipped at most once, and since tracing the faces initially and updating the tracing after each flipping are fast, the operations in the algorithm are easily implemented in polynomial time using the rotation scheme representation of an embedding.
\end{proof}

An example is shown in Figure \ref{fig:rsalg}, where we represent orientable embeddings of a graph as plane drawings with possible edge crossings.  With this representation, the rotation scheme just corresponds to the clockwise ordering at each vertex in the drawing, and we can trace faces in the usual way for a plane graph, except that we ignore edge crossings.
We show a complete run of the algorithm, which requires two iterations.
The tracing of the face $f$ is shown initially and after each iteration (in red, if color appears) and the edges $e_1$, $e_2$ used in the two iterations are labeled.

\begin{figure}[t]
\realline{\hfill%
	\includegraphics[scale=1.27]{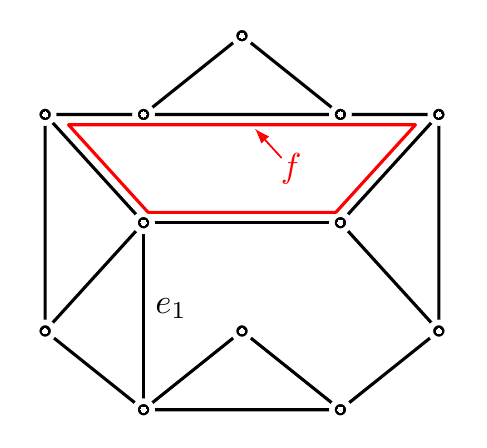}%
	\kern20pt\raise70pt\hbox{$\to$}\kern20pt%
	\includegraphics[scale=1.27]{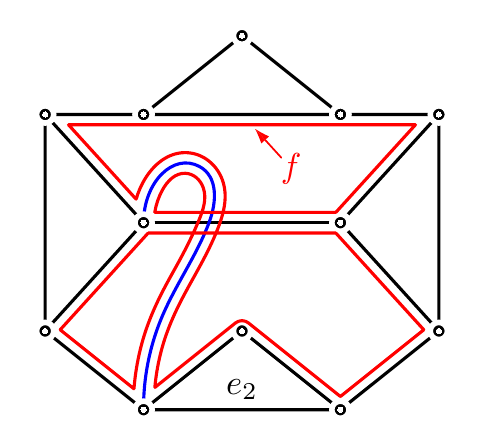}%
	\hfill}
\realline{\hfill%
	\raise70pt\hbox{$\to$}\kern20pt%
	\includegraphics[scale=1.27]{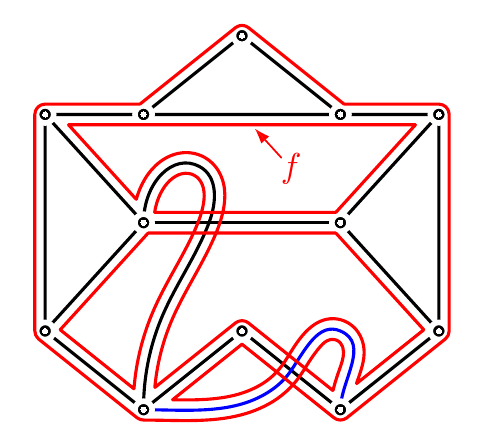}%
	\hfill}
\caption{Example of Algorithm \ref{thm:rsalg}.}
\label{fig:rsalg}
\end{figure}

\begin{corollary}\label{cor:rsmaxgenus}
 Every connected graph $G$ (loops and multiple edges allowed) has an orientable embedding such that there is a face $f$ whose boundary contains every edge, and the genus of the embedding is the maximum orientable genus of $G$.
\end{corollary}

\begin{proof}
 Apply Algorithm \ref{thm:rsalg} to $G$, beginning with a maximum genus orientable embedding of $G$.  From the parenthetical note in the proof of the algorithm, at each step the number of faces stays the same or decreases, so the genus of the embedding stays the same or increases.  Since we started with a maximum genus embedding and the genus does not decrease, the final embedding is a maximum genus embedding. 
\end{proof}

\begin{remark}
Often, an important consideration in determining reporter or scaffolding strand walks is assuring that the result is not knotted.  The authors of \cite{EPSetal17} observed that knotted walks can result from A-trails (non-crossing Eulerian circuits) in toroidal meshes, while \cite{M+pre} characterizes knotted and unknotted A-trails in toroidal meshes and \cite{MH17} gives an an approximation algorithm for unknotted walks in surface triangulations. In \cite{B+15} the authors restrict to spheres to assure unknotted routes.  We note here that Algorithm \ref{thm:rsalg} starts with an abstract graph and outputs a walk that is a facial walk of the graph embedded in an orientable surface.  If the surface is embedded in $3$-dimensional space, then this walk bounds a disk, and hence is unknotted when viewed as a curve in space.  This means that every graph has some embedding in $3$-dimensional space, in fact an embedding in some orientable surface, with an unknotted reporter strand walk.
\end{remark}

\section{NP-completeness of short reporter strands}
\label{sec:2c}

\subsection{Two decision problems}

Given that reporter (or scaffolding) strand walks exist, for experimental efficiency it is natural to seek a shortest such walk, i.e., one with as few edges as possible, and hence a minimum number of duplicated edges.  
However, in this section we show that the following decision version of finding $\srs(G)$, the length of a shortest reporter strand walk in $G$, is NP-complete.

\begin{decisionproblem}{Shortest Reporter Strand Walk (SRS Walk)}
Given $(G, k)$ where $G$ is a graph and $k$ is a nonnegative integer, is $\srs(G) \le k$?  In other words, does $G$ have a reporter strand walk of length at most $k$?  
\end{decisionproblem}

\noindent Thus, the \textsc{SRS Walk} problem asks whether $G$ has an orientable embedding with a facial walk that uses all edges and has length at most $k$.  A `yes' instance can be certified by giving a suitable embedding of $G$, so \textsc{SRS Walk} is in NP. The construction used here to prove that \textsc{SRS Walk} is NP-complete is relatively straightforward; it forms the first step towards a stronger NP-completeness result given in Section \ref{sec:3c}.

All walks from this point onwards (including paths and cycles) are directed walks.
Let $\rev(W)$ denote the reverse of a walk $W$.
For two walks $W_1$ and $W_2$, $W_1 \cc W_2$ denotes their concatenation, which is only defined if the last vertex of $W_1$ is the first vertex of $W_2$.
A walk is \emph{edge-spanning} if it uses every edge at least once, and \emph{edge-$2$-bounded} if it uses every edge at most twice. In any walk an edge used exactly once is a \emph{solo} edge, and an edge used exactly twice (whether in the same direction, or in opposite directions) is a \emph{double} edge.

For much of this section and the next we will be working with simple graphs.  Thus, we can uniquely identify an edge between $u$ and $v$ using the notation $uv$.  We can also describe walks just using sequences of vertices: then $uv$ also means a one-edge walk from $u$ to $v$, in that direction.
Using special notation to distinguish between the (undirected) edge $uv$ and the (directed) walk $uv$ would be unwieldy; we rely instead on context or explicit textual explanation.

We are interested in walks that can occur as face boundaries in an orientable graph embedding.
Loosely speaking, if a walk $W$ can be a face boundary in some orientable embedding of the graph, then we should be able to glue a facial disk to the graph, identifying its boundary with $W$, without preventing the neighborhood of any vertex from being an open disk in some  embedding, and without introducing nonorientability.  

Formalizing this objective yields two properties we desire in walks.  Given a walk $W$ in $G$ and $v \in V(G)$, let $\Rot_G(W,v)$ be the graph with vertex set $\ees_G(v)$, where we add one edge between edge-ends $d, d'$ for each time $W$ enters $v$ on $d$ and immediately leaves on $d'$, or vice versa.
In an embedding, for each $v$ the union of $\Rot_G(W,v)$ over all facial walks $W$ is a cycle on $\ees_G(v)$ describing the (undirected) rotation at $v$, so each $\Rot_G(W,v)$ is a subgraph of such a cycle.
Therefore, we say a closed walk $W$ is \emph{rotation-compatible in $G$} if for every $v \in V(G)$, $\Rot_G(W,v)$ is either a cycle with vertex set $\ees_G(v)$, or a union of vertex-disjoint paths.
Also, a walk is \emph{orientable} if it uses each edge at most once in each direction.  A walk that is rotation-compatible or orientable is edge-$2$-bounded.

A result of \v{S}koviera and \v{S}ir\'{a}\v{n} \cite[Prop.~1]{SS87} implies that a closed walk in $G$ occurs as a face boundary in some orientable embedding of $G$ if and only if it is orientable and rotation-compatible in $G$.
Loosely, if $W$ is orientable and rotation-compatible then each $\Rot_G(W,v)$ describes a partial rotation at $v$ that can be arbitrarily completed to a full rotation, giving a rotation scheme.
Thus, a reporter strand walk in $G$ is precisely a closed walk that is edge-spanning, orientable, and rotation-compatible in $G$.

There is a natural lower bound on the length of a reporter strand walk.  A \emph{Chinese postman walk} is an edge-spanning closed walk of minimum length.
A Chinese postman walk is edge-$2$-bounded, but need not be rotation-compatible or orientable.
Such walks were first considered by Guan \cite{Guan60}.
We let $\cp(G)$ denote the length of a Chinese postman walk in $G$.  Since every reporter strand walk is an edge-spanning closed walk, $\srs(G) \ge \cp(G)$.
Thus, we have another decision problem that may be regarded as a more specific version of the \textsc{SRS Walk} problem.

\begin{decisionproblem}{Chinese Postman Reporter Strand Walk (CPRS Walk)}
Given a graph $G$, is $\srs(G)=\cp(G)$?
In other words, does $G$ have a reporter strand walk that is also a Chinese postman walk?  (Such a walk is a \emph{Chinese postman reporter strand walk} or \emph{CPRS walk}.)
\end{decisionproblem}

Edmonds and Johnson \cite{EJ73} showed that $\cp(G)$ can be computed in polynomial time.  Thus, we can verify a `yes' instance of the \textsc{CPRS Walk} problem in polynomial time by checking that a given reporter strand walk (the certificate) has length $\cp(G)$.  This shows that \textsc{CPRS Walk} is in NP.  Moreover, every instance $G$ of \textsc{CPRS Walk} can be transformed in polynomial time to the instance $(G, \cp(G))$ of the \textsc{SRS Walk} problem, involving the same graph.  Then $G$ is a `yes' instance of \textsc{CPRS Walk} if and only if
$(G, \cp(G))$ is a `yes' instance of \textsc{SRS Walk}.  Therefore, if we show that \textsc{CPRS Walk} is NP-complete for a class of graphs, \textsc{SRS Walk} is also NP-complete for that class.

While the right side of Figure \ref{fig:rsfw} shows that $K_4$ has a CPRS walk (since $\cp(K_4)=8$), in an arbitrary graph a CPRS walk may not exist, i.e., a shortest reporter strand walk is not generally a Chinese postman walk.  For example, even the $2$-vertex theta graph in Figure \ref{fig:theta} has a shortest reporter strand walk of length 6, while a Chinese postman walk has length 4, and thus the theta graph does not have a CPRS walk.

\begin{figure}[t]
\realline{\hfill%
	\raise0pt\hbox{\includegraphics[scale=1.27]{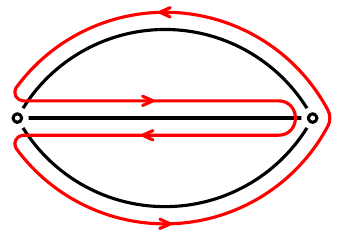}}%
	\hfill%
	\raise0pt\hbox{\includegraphics[scale=1.27]{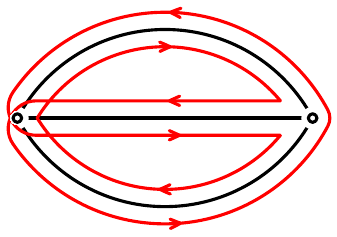}}%
	\hfill%
}
\caption{Chinese postman walk (left) and shortest reporter strand walk (right) in a theta graph.}
\label{fig:theta}
\end{figure}

Chinese postman walks in triangulations of the sphere are central to the strand routing algorithm of \cite{B+15}.  However, the walks produced by that algorithm may have retractions (defined below in Subsection \ref{subsec:cubic}), and so are not in general reporter strand walks according to our definition.
 
 We will prove that the problem \textsc{CPRS Walk} is NP-complete even when restricted to $2$-connected cubic planar graphs, and, in the next section, to $3$-connected cubic planar graphs.
We do this by reducing the hamilton cycle problem for $3$-connected cubic planar graphs, which is known to be NP-complete \cite{GJT76}, to \textsc{CPRS Walk}.

In this section and the following section, we work with cubic graphs and their subgraphs.

\subsection{Special properties of cubic graphs}
\label{subsec:cubic}

Both Chinese postman and reporter strand walks have special structures in $2$-connected cubic graphs.

First we consider Chinese postman walks.  Suppose $G$ is a $2$-connected cubic graph.
Any edge-spanning walk in $G$ uses all three edges at each vertex, so it must use each vertex at least twice, and hence it must have length at least $2|V(G)|$.  Thus, its length is $2|V(G)|$ if and only if it uses each vertex exactly twice, if and only if it contains exactly two solo edges and one double edge incident with each vertex.
Now, by a well-known result \cite{Petersen1891} of Petersen, $G$ has a perfect matching $M$.  If we replace each edge of $M$ in $G$ by two parallel edges to obtain $G'$, then $G'$ is eulerian, and an euler tour in $G'$ gives an edge-spanning walk in $G$ of length $2|V(G)|$. Therefore, a Chinese postman walk has length $2|V(G)|$, and hence uses two solo edges and one double edge at each vertex.
We summarize this as follows.

\begin{lemma}
\label{thm:cubiccp}
A closed walk in a $2$-connected cubic graph $G$ is a Chinese postman walk if and only if it is edge-spanning, edge-$2$-bounded, and its double edges form a perfect matching of $G$.
\end{lemma}

Now we consider reporter strand walks.
In cubic graphs, rotation-compatibility can be replaced by a simpler property.
A \emph{retraction} in a walk consists of an edge followed immediately by the same edge in the opposite direction.
For example, the walk on the left in Figure \ref{fig:theta} has a retraction.
A walk with no retractions is \emph{retraction-free}.
If a graph has no vertices of degree $1$, every rotation-compatible closed walk is retraction-free.
If a graph has no vertices of degree $4$ or more, every retraction-free edge-$2$-bounded closed walk is rotation-compatible.
Therefore, in a cubic graph a closed walk is rotation-compatible if and only if it is edge-$2$-bounded and retraction-free, giving the following.

\begin{lemma}
\label{thm:cubicrs}
A closed walk in a cubic graph $G$ is a reporter strand walk if and only if it is edge-spanning, orientable, and retraction-free.
\end{lemma}

\begin{corollary}
\label{thm:cubiccprs}
A closed walk in a $2$-connected cubic graph is a Chinese postman reporter strand (CPRS) walk if and only if it is edge-spanning, orientable, retraction-free, and the double edges form a perfect matching of $G$.
\end{corollary}

The two walks in $K_4$ shown in Figure \ref{fig:rsfw} illustrate these characterizations: both satisfy Lemma \ref{thm:cubicrs} and are reporter strand walks, and the one on the right satisfies Corollary \ref{thm:cubiccprs} and is a CPRS walk.

Suppose $W$ is a CPRS walk in $2$-connected cubic $G$, and $v \in V(G)$. Since $W$ is orientable, the double edge at $v$, call it $\db(Wv)$, is used in both directions by $W$, so one of the solo edges at $v$, call it $\ein(Wv)$, must be used by $W$ to enter $v$, and the other, $\eout(Wv)$, must be used by $W$ to leave $v$.  Since $W$ is retraction-free, it must use the edge sequences $\ein(Wv) \db(Wv)$ and $\db(Wv) \eout(Wv)$ to pass through $v$.

Therefore, we can reconstruct $W$ from the choice of double edges (which form a matching) and of orientations for the remaining solo edges (one entering, one leaving each vertex).  To consider possible CPRS walks we make such choices and try to trace $W$ by following the edge sequences $\ein(Wv) \db(Wv)$ and $\db(Wv)\eout(Wv)$ at $v$.  In general this \emph{tracing procedure} may fail by finding a closed walk that is not edge-spanning.  If this does not happen we obtain a CPRS walk.

A connected graph $H$ with two vertices $v_1, v_2$ of degree $2$ and all other vertices of degree $3$ is called an \emph{edge gadget}.  If $G$ is a graph disjoint from $H$ and $u_1u_2 \in E(G)$, then we say $J = (G-u_1u_2) \cup H \cup \{u_1v_1, u_2v_2\}$ is obtained by \emph{bisecting $u_1u_2$ in $G$ with $H$}.  The \emph{cubic completion} of $H$ is $H^+ = H \cup v_1v_2$.
We leave the proof of the following straightforward result to the reader.

\begin{lemma}
\label{thm:cubicop2}
Let $G$ be a cubic graph, and $H$ an edge gadget.  Construct $J$ by bisecting an edge of $G$ with $H$.  Then $J$ is cubic.
If $G$ and the cubic completion $H^+$ are both $2$-connected, planar, and simple, then
$J$ is $2$-connected, planar, and simple.
\end{lemma}

\subsection{The NP-completeness result}

\begin{figure}[t]
\realline{\hfill%
	\raise50pt\hbox{\includegraphics[scale=1.27]{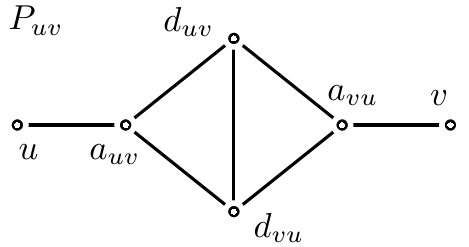}}%
	\hfill%
	\vbox{\hbox{\includegraphics[scale=1.27]{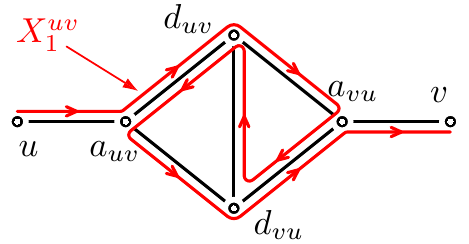}}
		\vskip10pt%
		\hbox{\includegraphics[scale=1.27]{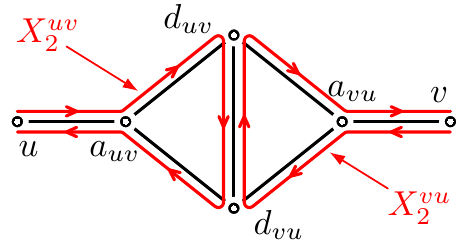}}%
	}
	\hfill%
}
\caption{Construction of $\gb$, and how a CPRS walk passes through $\gb_{uv}$.}
\label{fig:consgb}
\end{figure}

\begin{construction}
\label{thm:consgb}
Given a $3$-connected cubic planar simple graph $\ga$, construct a new graph $\gb$ by replacing each edge $uv$ by a subgraph $\gb_{uv}$ ($=\gb_{vu}$) consisting of a $4$-cycle $(\va(uv) \vd(uv) \va(vu) \vd(vu))$ on four new vertices and three additional edges $u \va(uv)$, $v \va(vu)$ and $\vd(uv) \vd(vu)$.  Note that order matters for subscripts in new vertex names.  We use $\autb(uv)$ ($=\autb(vu)$) to refer to the automorphism of $\gb_{uv}$ that swaps $\vd(uv)$ and $\vd(vu)$ and fixes the other vertices.
See Figure \ref{fig:consgb}.

\begin{claim}
The graph $\gb$ is a $2$-connected cubic planar simple graph.
\end{claim}
\end{construction}

\begin{proof}[Proof of claim]
The graph $\gb'_{uv} = \gb_{uv}-\{u,v\}$, obtained by removing $u$ and $v$ and all their incident edges from $\gb_{uv}$, is an edge gadget.  Moreover, $(\gb'_{uv})^+ \isom K_4$ is $2$-connected, planar, and simple.
Replacing $uv$ by $\gb_{uv}$ is equivalent to bisecting $uv$ with $\gb'_{uv}$, so the claim follows by repeated application of Lemma \ref{thm:cubicop2}.
\end{proof}

\begin{lemma}
\label{thm:walkgb}
Suppose we construct $\gb$ as in Construction \ref{thm:consgb}.
Let $\wb1uv = u \va(uv) \vd(uv) \va(vu) \vd(vu) \vd(uv)
	\- \va(uv) \vd(vu) \va(vu) v$ and
$\wb2uv = u \va(uv) \vd(uv) \vd(vu) \va(uv) u$.
Then a CPRS walk $W$ in $\gb$ must pass through each subgraph $\gb_{uv}$ in one of two ways,

{\leftskip=\parindent

\hangindent\parindent\noindent
(a) as a single walk
$\wb1uv$, $\autb(uv)(\wb1uv)$, $\revb(\wb1uv)$ or $\revb({\autb(uv)(\wb1uv)})$; or

\hangindent\parindent\noindent
(b) as two walks $\wb2uv$ and $\wb2vu$, or $\autb(uv)(\wb2uv)=\revb(\wb2uv)$ and $\autb(uv)(\wb2vu)=\revb(\wb2vu)$.

}
\end{lemma}

\begin{proof}
If $\vd(uv) \vd(vu)$ is a double edge of $W$, then $u \va(uv)$ and $v \va(vu)$ are also double edges.
The solo edges in $\gb_{uv}$ form a single cycle, which must be oriented consistently, as either $(\va(uv) \vd(uv) \va(vu) \vd(vu))$ or its reverse.  Applying the tracing procedure described above, (b) holds.

If $\vd(uv) \vd(vu)$ is not a double edge, the set of double edges in $\gb_{uv}$ is either $\{\va(uv) \vd(uv), \vd(vu) \va(vu)\}$ or $\{\va(uv) \vd(vu), \vd(uv) \va(vu)\}$.  By symmetry (from $\autb(uv)$) we may assume the former.  The solo edges form a single path
$u \va(uv) \vd(vu) \vd(uv) \va(uv) v$ which must be oriented in this direction or its reverse.  Applying the tracing procedure, (a) holds.
\end{proof}

Thus, Lemma \ref{thm:walkgb} says that up to symmetry or reversal a CPRS walk must pass through $\gb_{uv}$ in one of the ways shown on the right in Figure \ref{fig:consgb}.

\begin{figure}[t]
\realline{\hfill
	\raise57pt\hbox{\includegraphics[scale=1.27]{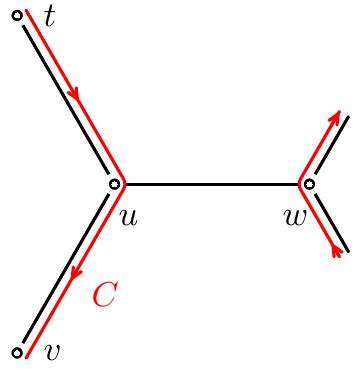}}%
	\hfill\hfill
	\raise120pt\hbox{$\to$}%
	\hfill\hfill
	\hbox{\includegraphics[scale=1.27]{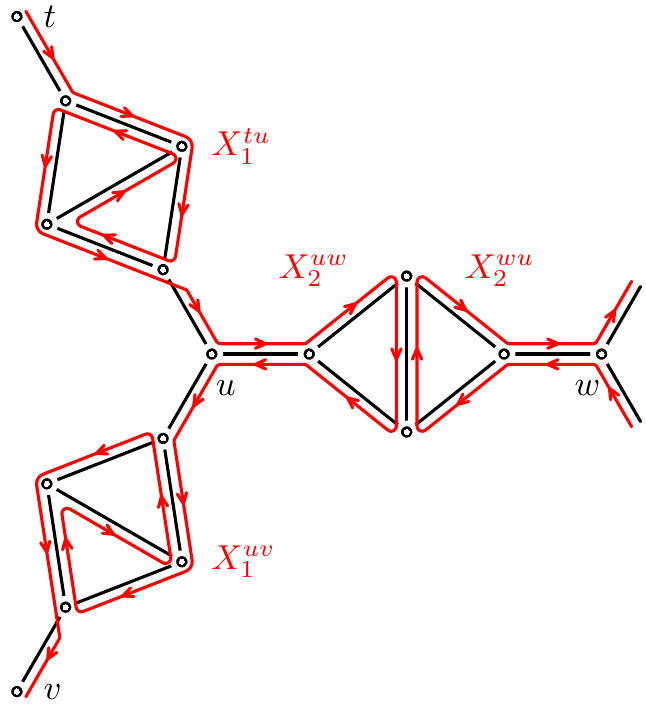}}%
	\hfill
}
\caption{Constructing a CPRS walk in $\gb$ from a hamilton cycle in $\ga$.}
\label{fig:conswb}
\end{figure}

\begin{proposition}
\label{thm:thm2c}
For $\ga$ and $\gb$ as in Construction \ref{thm:consgb}, the following are equivalent.

{\leftskip=\parindent

 \noindent (a) $\ga$ has a hamilton cycle.

 \noindent (b) $\gb$ has a Chinese postman reporter strand walk.

}
\end{proposition}

\begin{proof}
Suppose $\ga$ has a hamilton cycle, $C$. First, replace each (directed) edge $uv$ of $C$ by the walk $\wb1uv$.  This gives a walk in $\gb$ that uses every vertex of $N$ once.  Now for each edge $uw \in E(\ga)-E(C)$ splice $\wb2uw$ into this walk at $u$, and splice $\wb2wu$ into this walk at $w$.  The result is a CPRS walk in $\gb$.  See Figure \ref{fig:conswb}.

Conversely, suppose $\gb$ has a CPRS walk $W$.  By Lemma \ref{thm:walkgb}, at each $u \in V(\ga)$,
$\ein(Wu)$ belongs to some subwalk $\ww1tu = \wb1tu$ or $\autb(tu)(\wb1tu)$ of $W$,
$\eout(Wu)$ belongs to some $\ww1uv = \wb1uv$ or $\autb(uv)(\wb1uv)$,
and both occurrences of $\db(Wu)$ belong to some $\ww2uw = \wb2uw$ or $\revb(\wb2uw)$, where $t,v,w$ are the neighbors of $u$ in $N$.
Thus, deleting all subwalks $\ww2uw$ and replacing each subwalk $\ww1uv$ by the edge $uv$ of $\ga$ gives a hamilton cycle in $\ga$.
\end{proof}

Construction \ref{thm:consgb} therefore gives a polynomial time transformation from the hamilton cycle problem for $3$-connected cubic planar simple graphs, which is NP-complete \cite{GJT76}, to the \textsc{CPRS Walk} problem for $2$-connected cubic planar simple graphs.  This yields the following theorem.

\begin{theorem}
\label{thm:npc2c}
The problems \textsc{Shortest Reporter Strand Walk} and \textsc{Chinese Postman Reporter Strand Walk} are NP-complete for $2$-connected cubic planar simple graphs.
\end{theorem}

\section{A stronger NP-completeness result}
\label{sec:3c}

In this section we show that the problems \textsc{SRS Walk} and \textsc{CPRS Walk} are NP-complete even for $3$-connected planar graphs.

\subsection{Achieving $3$-connectedness}

While Section \ref{sec:2c} provides a simple proof of NP-completeness for the problems \textsc{SRS Walk} and \textsc{CPRS Walk}, the class of graphs that it uses does not have a stable $3$-dimensional structure, so they are not likely to occur in situations where we design a DNA molecule to have a specified geometric embedding in space.
In particular, the graphs $\gb$ produced by Construction \ref{thm:consgb} have connectivity $2$, while the graph formed by the edges of any polyhedron in $3$-dimensional space is $3$-connected.
Theorem \ref{thm:npc2c} leaves open the possibility that \textsc{SRS Walk} and \textsc{CPRS Walk} can be solved easily for $3$-connected graphs, or even that all $3$-connected graphs with more than two vertices have a CPRS walk.  Here we show that for $3$-connected graphs (in fact, $3$-connected cubic planar graphs) the problems \textsc{SRS Walk} and \textsc{CPRS Walk} are NP-complete, and hence unlikely to have polynomial-time solutions.
The construction in our proof yields arbitrarily large $3$-connected cubic planar graphs that do not have a CPRS walk.

First we modify the graph $\gb$ from Construction \ref{thm:consgb} to obtain a new graph $\gc$ with improved connectivity, in Construction \ref{thm:consgc}.
However, CPRS walks in $\gc$ do not necessarily correspond to CPRS walks in $\gb$, so later we further modify $\gc$ into a graph $\gd$ where we can control the CPRS walks so that they do correspond to CPRS walks in $\gb$, and hence to hamilton cycles in $\ga$.

Given a graph $G$ with a plane embedding, let $\cwn_G(u,v)$ denote the neighbor of $u$ that is immediately clockwise from $v$ in the rotation around $u$.

\begin{construction}
\label{thm:consgc}
Suppose we have $\ga$ and $\gb$ as in Construction \ref{thm:consgb}.
To construct $\gc$, take a plane embedding of $\ga$, and a corresponding plane embedding of $\gb$ in which each $4$-cycle $(\va(uv) \vd(uv) \va(vu) \vd(vu))$ is clockwise.
Replace each edge $\va(uv) \vd(uv)$ of $\gb$ by a path $\va(uv) \vb(uv) \vc(uv) \vd(uv)$ involving two new vertices $\vb(uv), \vc(uv)$.
Then incident to each vertex $\vc(uv)$ add a \emph{bracing edge}
$\vc(uv) \vb(vw)$ where $w = \cwn_\ga(v,u)$.
See Figure \ref{fig:consgc}.
\end{construction}

\begin{figure}[t]
\realline{\hfill%
	\includegraphics[scale=1.27]{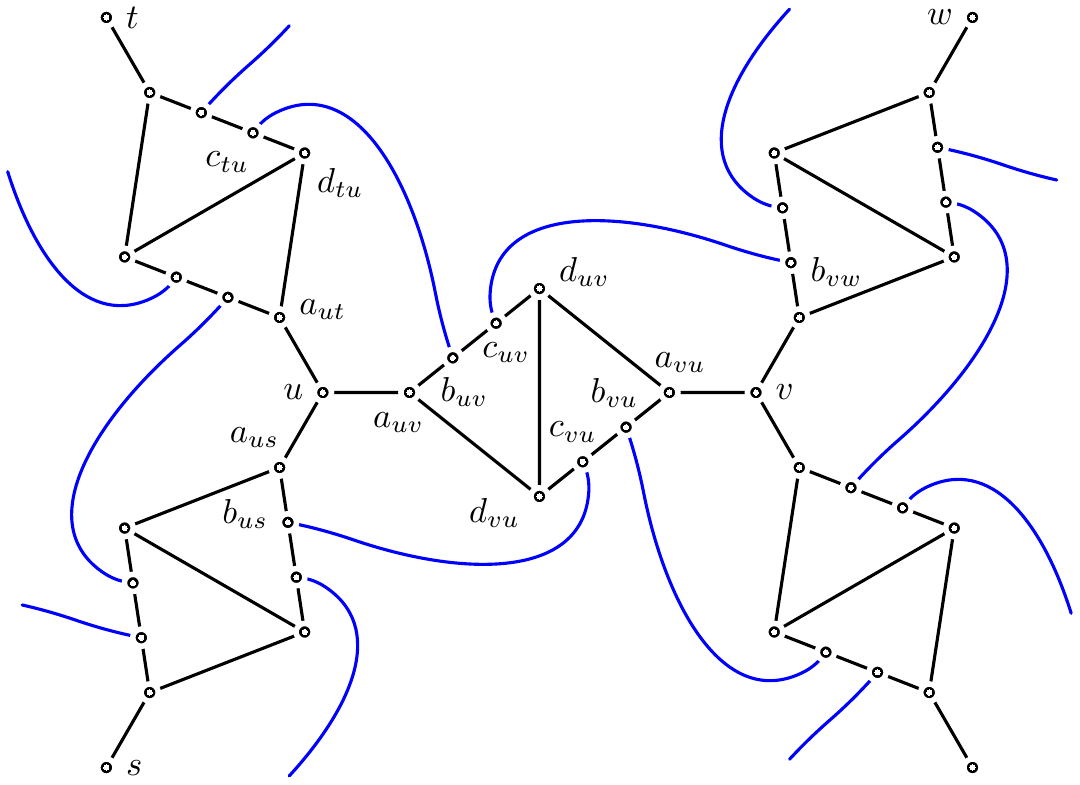}%
	\hfill}
\caption{Construction of $\gc$.}
\label{fig:consgc}
\end{figure}

Given a graph $G$, define a relation $\ethg$, or just $\eth$, on $V(G)$ by $u \eth v$ when there are three edge-disjoint $uv$-paths in $G$.
Therefore, by the edge version of Menger's Theorem (see \cite[Theorem 4.2.19]{West}), $u\eth v$ if and only if no set of fewer than $3$ edges separates $u$ and $v$.  It follows that $G$ is $3$-edge-connected precisely when all vertices of $G$ are $\eth$-equivalent.

\begin{lemma}\label{thm:e3}
$\eth$ is an equivalence relation.
\end{lemma}

\begin{proof} $\eth$ is reflexive (take three copies of the trivial walk at a vertex) and clearly symmetric; we must show it is transitive.  Suppose that $u\eth v$ and $v\eth w$.  If we do not have $u\eth w$ then some set of fewer than $3$ edges separates $u$ and $w$.  But then this set either separates $u$ and $v$, contradicting $u\eth v$, or $v$ and $w$, contradicting $v\eth w$.  Hence, $u\eth w$.
\end{proof}

\begin{lemma}
\label{thm:gc3c}
The graph $\gc$ is $3$-connected, planar, and simple.
\end{lemma}

\begin{proof}
Clearly $\gc$ is planar and simple (see Figure \ref{fig:consgc}).
For cubic graphs such as $\ga$ and $\gc$, $3$-connected\-ness is equivalent to $3$-edge-connected\-ness, which is equivalent to showing that all vertices are $\eth$-equivalent.

Vertices of $\gc$ are either \emph{original} vertices, namely vertices of $\ga$, or \emph{new} vertices, added by Constructions \ref{thm:consgb} and \ref{thm:consgc}.
If $u$ and $v$ are original vertices then there are three edge-disjoint $uv$-paths in $\ga$, which easily provide three edge-disjoint $uv$-paths in $\gc$.  Hence all original vertices are $\ethgc$-equivalent.
So it suffices to show that each new vertex is $\ethgc$-equivalent to some original vertex.

Suppose that in the plane embedding of $\ga$, the neighbors of $u$ are $s, t, v$ in clockwise order.  The following paths from new vertices of $\gc$ to the original vertex $u$ (see Figure \ref{fig:consgc}) show that $\va(uv)$, $\vb(uv)$ and $\vd(uv)$ are $\ethgc$-equivalent to $u$:

{\advance\leftskip by\parindent

\noindent $\va(uv)u$-paths:\quad
$\va(uv) u$,\quad
$\va(uv) \vb(uv) \vc(tu) \vd(tu) \va(ut) u$,\quad
$\va(uv) \vd(vu) \vc(vu) \vb(us) \va(us) u$.

\noindent $\vb(uv)u$-paths:\quad
$\vb(uv) \va(uv) u$,\quad
$\vb(uv) \vc(uv) \vd(uv) \vd(vu) \vc(vu) \vb(us) \va(us) u$,\quad
$\vb(uv) \vc(tu) \vd(tu) \va(ut) u$.

\noindent $\vd(uv)u$-paths:\quad
$\vd(uv) \vd(vu) \va(uv) u$,\quad
$\vd(uv) \vc(uv) \vb(uv) \vc(tu) \vd(tu) \va(ut) u$,\quad
$\vd(uv) \va(vu) \vb(vu) \vc(vu) \vb(us) \va(us) u$.

}

\noindent
Rather than $\vc(uv)$ it is more convenient to show that $\vc(vu)$ is $\ethgc$-equivalent to $u$:

{\advance\leftskip by\parindent

\noindent $\vc(vu)u$-paths:\quad
$\vc(vu) \vd(vu) \va(uv) u$,\quad
$\vc(vu) \vb(us) \va(us) u$,\quad
$\vc(vu) \vb(vu) \va(vu) \vd(uv) \vc(uv) \vb(uv) \vc(tu) \vd(tu) \va(ut) u$.

}

\noindent
Since every new vertex is $\va(uv)$, $\vb(uv)$, $\vd(uv)$ or $\vc(vu)$ for some choice of $u$ and $v$, every new vertex is $\ethgc$-equivalent to an original vertex, as required.
\end{proof}

\subsection{Controlling walks}

A connected graph $H$ with three vertices $v_1, v_2, v_3$ of degree $2$ and all other vertices of degree $3$ is called a \emph{vertex gadget}.
If $G$ is a graph disjoint from $H$ and $u \in V(G)$ has degree $3$ with neighbors $u_1, u_2, u_3$, then we say the graph
$J = (G-u) \cup H \cup \{u_1v_1, u_2v_2, u_3v_3\}$
is obtained by \emph{replacing $u$ in $G$ by $H$}.
The \emph{cubic completion} of $H$ is $H^+ = H \cup \{vv_1, vv_2, vv_3\}$ where $v$ is a new vertex.  We leave the proof of the following straightforward result to the reader.

\begin{lemma}
\label{thm:cubicop3}
Let $G$ be a cubic graph, and $H$ a vertex gadget.  Construct $J$ by replacing a vertex of $G$ with $H$.  Then $J$ is cubic.
If $G$ and the cubic completion $H^+$ are both $3$-connected, planar, and simple, then
$J$ is $3$-connected, planar, and simple.
\end{lemma}

Now we construct subgraphs in which the route taken by a CPRS walk is constrained in various ways.

\begin{figure}[t]
\realline{\hfill%
	\includegraphics[scale=1.27]{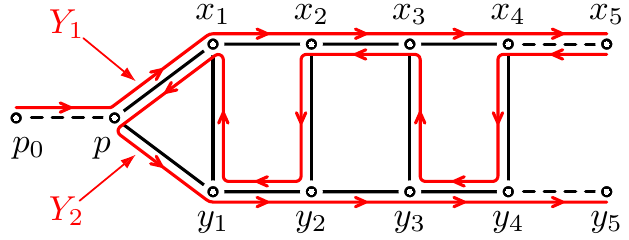}%
	\hfill}
\caption{Vertex gadget $\gs$, and how a CPRS walk passes through it.}
\label{fig:consgs}
\end{figure}

Let $\gs$ be the vertex gadget shown (with additional incident edges $\vpz \vp, \vxd \vxe, \vyd \vye$) in Figure \ref{fig:consgs}.
Let $\auts$ be the automorphism of $\gs$ that swaps the two paths $\vxa \vxb \vxc \vxd$ and $\vya \vyb \vyc \vyd$ while fixing $\vp$.
Note that the cubic completion $\gs^+$ is $3$-connected (to see this, observe that for every $v \in V(\gs^+)$, $\gs^+-v$ has a hamilton cycle and is therefore $2$-connected).  Also, $\gs^+$ is planar and simple.

\begin{lemma}
\label{thm:walkgs}
Suppose the vertex gadget $\gs$ described above is an induced subgraph of a $2$-connected cubic graph $G$.
Let $\ws1 = \vp \vxa \vxb \vxc \vxd$ and
$\ws2 = \vxd \vyd \vyc \vxc \vxb \vyb \vya \vxa \vp \vya \vyb \vyc
	\vyd$.
If $W$ is a CPRS walk in $G$ then $W$ passes through $\gs$ and its incident edges as two walks, either
$\vpz \vp \cc \ws1 \cc \vxd \vxe$ and
$\vxe \vxd \cc \ws2 \cc \vyd \vye$, or
$\vpz \vp \cc \auts(\ws1) \cc \vyd \vye$ and
$\vye \vyd \cc \auts(\ws2) \cc \vxd \vxe$,
or reversing both walks in one of these pairs.
\end{lemma}

\begin{proof}
Suppose first that $\vpz \vp$ is a double edge in $W$.
If $\vxa \vya$ is not a double edge then $\vxa \vxb$ and $\vya \vyb$ are double edges.  We have a triangle $(\vp \vxa \vya)$ of solo edges; we may assume its edges are oriented in that direction by $W$.
We have another path of solo edges $\vxc \vxb \vyb \vyc$.  If this is oriented as $\vyc \vyb \vxb \vxc$ then then the tracing procedure fails by finding a $4$-cycle $(\vxa \vya \vyb \vxb)$.  So it is oriented as $\vxc \vxb \vyb \vyc$.  If $\vxc \vyc$ is a double edge then the tracing algorithm fails by finding a $6$-cycle $(\vxa \vya \vyb \vyc \vxc \vxb)$.
Thus, $\vxc \vyc$ is a solo edge, it must be oriented as $\vyc \vxc$, $\vxc \vxd$ and $\vyc \vyd$ are double edges, and $\vxd \vyd$ is a solo edge.  If $\vxd \vyd$ is oriented as $\vxd \vyd$, then the tracing procedure fails by finding a $4$-cycle $(\vyc \vxc \vxd \vyd)$, and if it is oriented as $\vyd \vxd$, then the tracing algorithm fails by finding an $8$-cycle $(\vxa \vya \vyb \vyc \vyd \vxd \vxc \vxb)$.

If $\vxa \vya$ is a double edge we have a path of solo edges $\vyb \vya \vp \vxa \vxb$ which without loss of generality is oriented in that direction.
If $\vxb \vyb$ is a double edge, then the tracing procedure fails by finding the $4$-cycle $(\vyb \vya \vxa \vxb)$.
So $\vxb \vyb$ is a solo edge, it must be oriented as $\vxb \vyb$, $\vxb \vxc$ and $\vyb \vyc$ are double edges, and $\vxc \vyc$ is a single edge.  If $\vxc \vyc$ is oriented as $\vxc \vyc$ then our tracing algorithm fails by finding a $6$-cycle $(\vxc \vyc \vyb \vya \vxa \vyb)$, and if it is oriented as $\vyc \vxc$ then the tracing algorithm fails by finding a $4$-cycle $(\vyc \vxc \vxb \vyb)$.

Therefore, $\vpz \vp$ is a solo edge; without loss of generality, $\vpz \vp = \ein(W\vp)$.  By symmetry (from $\auts$) we may assume that $\db(W\vp) = \vp \vxa$.
Then $\vya \vyb$, $\vxb \vxc$, $\vyc \vyd$ and $\vxd \vxe$ must all be double edges.  The solo edges form a single path which is oriented
$\vpz \vp \vya \vxa \vxb \vyb \vyc \vxc \vxd \vyd \vye$.
Now the tracing procedure gives
$\vpz \vp \cc \ws1 \cc \vxd \vxe$ and
$\vxe \vxd \cc \ws2 \cc \vyd \vye$.
\end{proof}

Thus, Lemma \ref{thm:walkgs} says that up to symmetry or reversal a CPRS walk must pass through $\gs$ as shown in Figure \ref{fig:consgs}.
Loosely, $\gs$ acts like a vertex, in that a CPRS walk passes through it as two walks of the form (entering solo edge)(intermediary edges)(exiting double edge) and (entering double edge)(intermediary edges)(exiting solo edge), but with a restriction: the edge $\vpz \vp$ must be a solo edge.

\begin{figure}[t]
\realline{\hfill%
	\includegraphics[scale=1.27]{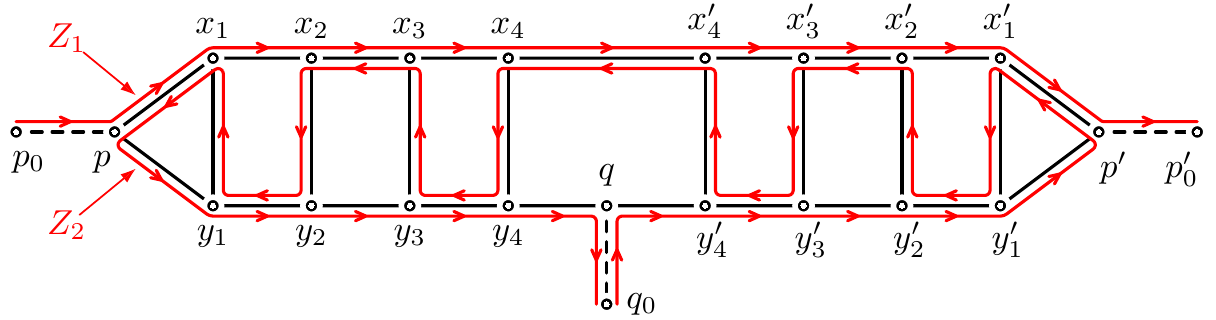}%
	\hfill}
\caption{Vertex gadget $\gt$, and how a CPRS walk passes through it.}
\label{fig:consgt}
\end{figure}

Now we build a larger vertex gadget.  Let $\gs'$ be a copy of $\gs$, with a plane embedding that is the mirror image of the embedding of $\gs$ in Figure \ref{fig:consgs}.  Let $\vp'$ in $\gs'$ correspond to $\vp$ in $\gs$, and so on.
Let $\gt = \gs \cup \gs' \cup \{\vxd \vxd', \vyd \vq, \vyd' \vq\}$ where $\vq$ is a new vertex.  Then $\gt$ is a vertex gadget.  Note that $\gt^+$ can be considered as obtained from $K_4$ by replacing two vertices by copies of $\gs$, so by Lemma \ref{thm:cubicop3} applied twice, $\gt^+$ is $3$-connected, planar, and simple.

\begin{lemma}
\label{thm:walkgt}
Suppose the vertex gadget $\gt$ described above is an induced subgraph of a $2$-connected cubic graph $G$, with incident edges $\vpz \vp$, $\vq \vqz$ and $\vp' \vpz'$.
Let $\wto1 = \ws1 \cc \vxd \vxd' \cc \revb(\ws1')$ and
$\wto2 = \vq \vyd' \cc \revb(\ws2') \cc \vxd' \vxd \cc \ws2 \cc
	\vyd \vq$.
If $W$ is a CPRS walk in $G$ then $W$ passes through $\gt$ and its incident edges as two walks,
either
$\vpz \vp \cc \wto1 \cc \vp' \vpz'$ and
$\vqz \vq \cc \wto2 \cc \vq \vqz$,
or reversing both of these walks.
\end{lemma}

\begin{proof}
Applying Lemma \ref{thm:walkgs} to both $\gs$ and $\gs'$, $\vpz \vp$ and $\vp' \vpz'$ are solo edges.
The perfect matching of double edges of $W$ must have an odd number of edges leaving the odd set $V(\gt)$, so $\vqz \vq$ must be a double edge.
Therefore, $\vq \vyd$ and $\vq \vyd'$ are solo edges.
Now Lemma \ref{thm:walkgs}, applied to both $\gs$ and $\gs'$, gives the result.
\end{proof}

Thus, Lemma \ref{thm:walkgt} says that up to reversal (or, equivalently, up to the automorphism of $\gt$ swapping $p$ and $p'$) a CPRS walk must pass through $\gt$ as shown in Figure \ref{fig:consgt}.

\subsection{NP-completeness for $3$-connected cubic planar graphs}

\begin{construction}
\label{thm:consgd}
Suppose we have $\ga$, $\gb$ and $\gc$ as in Constructions \ref{thm:consgb} and \ref{thm:consgc}.
For each vertex $\vb(uv)$ take a copy $\gt_{uv}$ of $\gt$, where $\vp_{uv}, \vq_{uv}, \vp'_{uv}, \wt1uv, \wt2uv$ correspond to $\vp, \vq, \vp', \wto1, \wto2$ in $\gt$, respectively.
Construct $\gd$ by replacing each vertex of the form $\vb(uv)$ in $\gc$ by $\gt_{uv}$, so that if $\vb(uv)$ is adjacent to $\va(uv), \vc(uv), \vc(tu)$ then the edges incident with $\gt_{uv}$ are
$\va(uv) \vp_{uv}$,
$\vq_{uv} \vc(tu)$ and
$\vp'_{uv} \vc(uv)$.

\begin{claim}
The graph $\gd$ is a $3$-connected cubic planar simple graph.
\end{claim}
\end{construction}

\begin{proof}[Proof of claim]
As noted above, $\gt^+$ is a $3$-connected, planar and simple, and so is $\gc$ by Lemma \ref{thm:gc3c}.  The claim follows by repeated application of Lemma \ref{thm:cubicop3}.
\end{proof}

\begin{proposition}
\label{thm:thm3c}
For $\ga$, $\gb$, $\gc$ and $\gd$ as in Constructions \ref{thm:consgb}, \ref{thm:consgc} and \ref{thm:consgd}, the following are equivalent.

{\leftskip=\parindent

 \hangindent\parindent
 \noindent (a) $\ga$ has a hamilton cycle.

 \hangindent\parindent
 \noindent (b) $\gb$ has a Chinese postman reporter strand walk.

 \hangindent\parindent
 \noindent (c) $\gb$ has a Chinese postman reporter strand walk using every edge of the form $\va(uv) \vd(uv)$ as a solo edge.

 \hangindent\parindent
 \noindent (d) $\gd$ has a Chinese postman reporter strand walk.

}
\end{proposition}

\begin{proof}
By Proposition \ref{thm:thm2c}, (a) $\siff$ (b).  Clearly (c) $\simplies$ (b).  Suppose (b) holds and we have a CPRS walk $W$ in $\gb$.
Suppose some $\va(uv) \vd(uv)$ is not a solo edge of $W$.
By Lemma \ref{thm:walkgb}, $W$ must use $\wb1uv$ or its reverse; replacing this by $\autb(uv)(\wb1uv)$ or its reverse we still have a CPRS walk, and now $\va(uv) \vd(uv)$ (and also $\va(vu) \vd(vu)$) is a solo edge.  Applying this to all $\va(uv) \vd(uv)$ that are not solo edges, we obtain a CPRS walk $W'$ satisfying (c).  Thus, (b) $\simplies$ (c).

So now we show that (c) $\siff$ (d).
Suppose that (c) holds, with a walk $W$ using each $\va(uv) \vd(uv) \in E(\gb)$ as a solo edge.
The bracing edge of $\gc$ incident with $\vc(uv)$ has the form $\vc(uv) \vb(vw)$, where $w$ follows $u$ in clockwise order around $v$ in $\ga$.
Replace each directed edge $\va(uv) \vd(uv)$, or its reverse, in $W$ by a walk in $\gd$ according to the following rules:

\smallskip
$W$ uses $\va(uv) \vd(uv)$, $\va(vw) \vd(vw)$:
$\va(uv) \vd(uv)$ $\to$
$T^{00}_{uv} = \va(uv) \vp_{uv} \cc \wt1uv \cc \vp'_{uv} \vc(uv)
	\vq_{vw} \cc
	\wt2vw \cc \vq_{vw} \vc(uv) \vd(uv)$.

\smallskip
$W$ uses $\va(uv) \vd(uv)$, $\vd(vw) \va(vw)$:
$\va(uv) \vd(uv)$ $\to$
$T^{01}_{uv} = \va(uv) \vp_{uv} \cc \wt1uv \cc \vp'_{uv} \vc(uv)
	\vq_{vw} \cc
	\revb(\wt2vw) \cc \vq_{vw} \vc(uv) \vd(uv)$.

\smallskip
$W$ uses $\vd(uv) \va(uv)$, $\va(vw) \vd(vw)$:
$\vd(uv) \va(uv)$ $\to$
$T^{10}_{uv} = \vd(uv) \vc(uv) \vq_{vw} \cc \wt2vw \cc \vq_{vw}
	\vc(uv)
	\vp'_{uv} \cc \revb(\wt1uv) \cc \vp_{uv} \va(uv)$.

\smallskip
$W$ uses $\vd(uv) \va(uv)$, $\vd(vw) \va(vw)$:
$\vd(uv) \va(uv)$ $\to$
$T^{11}_{uv} = \vd(uv) \vc(uv) \vq_{vw} \cc \revb(\wt2vw) \cc
	\vq_{vw} \vc(uv)
	\vp'_{uv} \cc \revb(\wt1uv) \cc \vp_{uv} \va(uv)$.

\smallskip
\noindent
The rules guarantee that in each $\gt_{uv}$ we use both $\wt1uv$ and $\wt2uv$, or both $\revb(\wt1uv)$ and $\revb(\wt2uv)$.
Therefore, the result is a CPRS walk $W'$ in $\gd$.  Thus, (d) holds.

Conversely, suppose (d) holds, so $\gd$ has a CPRS walk $W$. Consider each $\va(uv) \vd(uv) \in E(\gb)$ and the corresponding bracing edge $\vc(uv) \vb(vw) \in E(\gc)$.
 Applying Lemma \ref{thm:walkgt} to $\gt_{uv}$ and $\gt_{vw}$, we see that $W$ must either travel from $\va(uv)$ to $\vd(uv)$ along  $T^{00}_{uv}$ or $T^{01}_{uv}$ from above, or travel from $\vd(uv)$ to $\va(uv)$ along $T^{10}_{uv}$ or $T^{11}_{uv}$.
In the former case, replace this subwalk of $W$ by the edge $\va(uv) \vd(uv)$ of $\gb$; in the latter case replace it by $\vd(uv) \va(uv)$.  Making all such replacements gives a CPRS walk $W'$ in $\gb$ in which each $\va(uv) \vd(uv)$ is a solo edge.  Thus, (c) holds.
\end{proof}

Constructions \ref{thm:consgb}, \ref{thm:consgc} and \ref{thm:consgd} therefore give a polynomial time transformation from the hamilton cycle problem for $3$-connected cubic planar simple graphs to the \textsc{CPRS Walk} problem for the same family of graphs.
Applying these constructions to nonhamiltonian $3$-connected cubic planar graphs $\ga$ proves the existence of arbitrarily large $3$-connected cubic planar simple graphs $\gd$ with no CPRS walk (or we can construct small examples of such graphs easily using vertex gadgets $\gs$ and $\gt$).
Our final theorem also follows immediately.

\begin{theorem}
\label{thm:npc3c}
The problems \textsc{Shortest Reporter Strand Walk} and \textsc{Chinese Postman Reporter Strand Walk} are NP-complete for $3$-connected cubic planar simple graphs.
\end{theorem}

\section{Conclusion}
\label{sec:conclusion}

This application brings to light a new, natural area of investigation in topological graph theory,  edge-outer embeddability, which seems quite rich in attractive questions and new directions:

\begin{enumerate}

\item
The algorithm in Section \ref{sec:shortproof}  provides a fast routing solution that is within 100\% of optimal (at most twice the length).  Is there a polynomial-time algorithm that will return a reporter strand walk that is within a smaller percentage of minimum length?  A related result appears in \cite{MH17}, where they give a cubic-time $\frac {5}{3}$-approximation algorithm in the special case that the graph is a triangulation of an orientable surface.

\item Can we extend Corollary \ref{cor:rsmaxgenus} to say more about the genus range of embeddings that yield reporter strand walks, or reporter strand walks of minimum length?  Are these ranges intervals?

\item Are there classes of graphs where it is polynomial-time to find a minimum length reporter strand walk?  Eulerian graphs are one such class. We have shown that the problem is NP-hard for 3-connected graphs, but can it be solved in polynomial time for graphs with higher connectivity?  

\item What pragmatic approaches might there be to finding suitable scaffolding strand routes, albeit possibly with restrictions or other design costs?  One such approach is provided by \cite{B+15}, which describes a strand routing design algorithm using an A-trail heuristic that performs well on reasonably sized triangulations of the sphere, provided that some `double-width' edges (using two double helices) are acceptable in the final product.   Another approach may be found in \cite{V+17}, which gives a fast algorithm, but essentially makes all of the edges `double-width'. Other methods of efficiently determining suitable routes with reasonable design trade-offs would help advance the field of DNA origami. 

\end{enumerate}

\end{document}